\documentclass[aap,MSNbibl,citesort,dvips]{arximspdf}

\doi{10.1214/11-AAP810} 
\volume{22}
\issue{4}
\pubyear{2012}
\firstpage{1712}
\lastpage{1727}

\makeatletter
\newtheorem{thmm}{Theorem}[section]
\newtheorem{claim}{Claim}
\newtheorem{lem}[thmm]{Lemma}
\newproclaim{remark}{Remark}

\newcommand{\R}{\mathbb{R}} 
\newcommand{\cL}{\mathcal{L}}
\newcommand{\PP}{\mathbb{P}} 
\newcommand{\TV}{\mathrm{TV}}
\newcommand{\mix}{\mathrm{mix}}
\newcommand{\eqref}[1]{(\ref{#1})}

\makeatother

\begin{document}
\begin{frontmatter}

\title{Total variation bound for Kac's random walk}
\runtitle{Total variation mixing time of Kac's random walk}

\begin{aug}
\author[A]{\fnms{Yunjiang} \snm{Jiang}\corref{}\ead[label=e1]{jyj@math.stanford.edu}\thanksref{t1}}
\thankstext{t1}{Research supported in part by an NSF graduate fellowship.}
\runauthor{Y. Jiang}
\affiliation{Stanford University}
\address[A]{126 Blackwelder Ct., Apt 605A\\
Stanford, California 94305\\
USA\\
\printead{e1}} 
\end{aug}

\received{\smonth{12} \syear{2009}}
\revised{\smonth{2} \syear{2011}}

%
\begin{abstract}
We show that the classical Kac's random walk on $(n-1)$-sphere
$S^{n-1}$ starting from the
point mass at $e_1$ mixes in $\mathcal{O}(n^5 (\log n)^3)$ steps in
total variation
distance. The main argument uses a truncation of the running density
after a burn-in period, followed by $\mathcal{L}^2$ convergence using
the spectral gap information derived by other authors. This improves
upon a previous bound by Diaconis and Saloff-Coste of order $\mathcal
{O}(n^{2n})$.
\end{abstract}


\end{frontmatter}
%

\section{Introduction}\label{sec1}

Consider $n$ particles on $\R$ making random pairwise collisions, in
such a way that the total kinetic energy is conserved. Since there is
randomness involved, the situation is typically modeled by a Markov
chain. Two natural questions are how would the particles be distributed
in equilibrium and whether such equilibrium distribution is unique. And
once these are answered, one would also like to know how long it takes
for the particles to reach this equilibrium distribution. Of course
these questions would depend on the mathematical models we choose to
describe the system.

Mark Kac proposed the following toy model of one-dimensional Boltzmann
gas dynamics that captures the above description (for historical
development, see~\cite{PDLSC,Kac}):
For the $n$ particles on $\R$, we can represent their velocities
$(v_1, \ldots, v_n)$ as a point on the unit sphere $S^{n-1}$ after
normalization so that
\[
\sum_{i=1}^n v_i^2 = 1.
\]
Conservation of kinetic energy (assuming $0$ potential energy) in the
gas dynamics is equivalent to $(v_1(t), \ldots, v_n(t))$ staying on
$S^{n-1}$ for all $t \ge0$. We will not introduce momentum
conservation in our model, because that will force the collision to
be\vadjust{\goodbreak}
inelastic (see second paragraph below), and reduces the model to a
discrete Markov chain such as the random transposition walk on $S_n$.
But when the particles live in $\R^3$, momentum conservation becomes
quite interesting (see~\cite{momentum}). The technique in this paper
might be applicable to that model as well.

Each time there is a collision, it occurs with probability $1$ between
no more than two particles, which corresponds to choosing two distinct
coordinate directions $x_i, x_j$ and rotating $S^{n-1}$ along the
2-plane $x_i \wedge x_j$ by some angle~$\theta$. Notice that $\sum_k
v_k^2 = 1$ both before and after the collision, since the sum $v_i^2 +
v_j^2$ is not affected by the rotation along the $i,j$ plane and all
the other velocities stay the same.

By disregarding the position information of the particles (which have
to be confined in some compact domain, for example $S^1$, else they
will eventually run off to infinity), each collision occurs between any
pair of the particles with equal probability $\frac{1}{{n \choose2}}$.
The rotation angle $\theta$ can be chosen from some distribution on
$[0,2\pi)$, which physically is a measure of the elasticity of the
collision; for example, inelastic collision in $\R$ will correspond to
a distribution of $\theta$, that is, a delta measure concentrated at
$\pi
$. In this paper, we will assume that $\theta$ is uniformly distributed
on $[0, 2\pi)$.

Thus we obtain a discrete-time Markov chain on $S^{n-1}$ with
transition kernel given by, for $f\dvtx  S^{n-1} \to\R$ continuous, and $x
\in S^{n-1}$,
\begin{equation} \label{Kac kernel}
(Kf) (x)= \frac{1}{{n \choose2}} \sum_{i\neq j}^n \int_0^{2\pi}
f(R(i,j;\theta) x) \frac{1}{2\pi} \,d\theta,
\end{equation}
where $R(i,j;\theta)$ denotes the rotation along the oriented $i\wedge
j$ plane by the angle~$\theta$, and $R(i,j;\theta) x$ signifies the
usual action of $SO(n)$ on $S^{n-1}$. By transposing, $K$ defines a map
from the set of probability measures on $S^{n-1}$ to itself, since
\mbox{$K(1) = 1$}.

Since the Lie group $SO(n)$ acts on itself, one can also define Kac's
walk~$\tilde{K}$ on $SO(n)$, given on test functions by
\begin{equation} \label{Kac on SO(n)}
(\tilde{K} f)(A) = \frac{1}{{n \choose2}} \sum_{i\neq j}^n \int
_0^{2\pi} f(R(i,j;\theta) A) \frac{1}{2\pi} \,d\theta,
\end{equation}
where $A$ is any element of $SO(n)$.

It is easy to check that $U_{n-1}$, the uniform distribution on
$S^{n-1}$, is a~stationary distribution for $K$: for each summand
$K_{i,j}$ (without $\frac{1}{{n \choose2}}$ in \eqref{Kac kernel}),
we have
\begin{eqnarray*}
U_{n-1} (K_{i,j} f) &= &\int_{S^{n-1}} (K_{i,j} f)(x) U_{n-1}(dx) \\
&=& \int_{S^{n-1}} \biggl(\int_0^{2\pi} f(R(i,j; \theta) x) \frac{1}{2\pi}
\,d\theta\biggr) U_{n-1}(dx) \\
&=& \int_{S^{n-1}} \frac{1}{2\pi} \biggl(\int_0^{2\pi} f(x) \,d\theta\biggr)U_{n-1}(
R(i,j;-\theta) \,dx) \\
&=& \int_{S^{n-1}} f(x) U_{n-1}( d x)
\end{eqnarray*}
using a change of variable formula and the fact that $U_{n-1}$ is
invariant under rotations. This establishes that $U_{n-1} K_{i,j} =
U_{n-1}$ for all $i \neq j$. Thus their average $U_{n-1} K = U_{n-1}$
as well.

By a similar argument, or more generally from the theory of random
walks on compact groups, we also deduce that the Haar measure is the
stationary distribution of $\tilde{K}$.

We further claim that the Markov chain defined by $K$ is aperiodic
because once a point is reached, it can be reached in the next step
with positive probability density for any rotation. It is also
irreducible since along a~sequence of rotations $(i_1 \wedge i_2,
\ldots
, i_k \wedge i_{k+1})$ that form a connected spanning graph in~$K_n$,
the complete graph on $n$ vertices, one can transport any point on~$S^{n-1}$
to any other point with positive probability density; such
sequence of rotations certainly occur with positive probability. In
fact, by a slightly more involved argument using Hurwitz factorization
of $SO(n)$ in terms of Givens' rotations~\cite{PDLSC}, one can show
that Kac's random walk on $SO(n)$ is also irreducible, which certainly
implies irreducibility on $S^{n-1}$ since the latter is a projection of
the former. Furthermore, both chains are recurrent because the state
space is compact. Thus by convergence theory of Harris chains, we know
that with any initial distribution $\mu$ on $S^{n-1}$,
\[
\lim_{l \to\infty} \mu K^l (A) - U_{n-1}(A) = 0
\]
uniformly in $A \subset S$. This implies convergence in total variation
distance by definition.

Using the $\cL^2$ theory of discrete-time Markov chains, it can be
shown that if the starting distribution $\mu$ is in $\cL
^2(S^{n-1},U_{n-1})$, then we get the following convergence bound:
\[
\| \mu K^l - U_{n-1}\|_{\TV} < \|\mu- 1\|_{\cL^2} \biggl(1-\frac{1}{2n}\biggr)^l
\]
by the result in~\cite{Kac} and~\cite{Maslen}, which show that the
spectral gap of $K$ is given by $\frac{n+2}{2n(n-1)}$ for $n\ge 2$.
See also~\cite{Janvresse} for an earlier Martingale argument to get $\Omega(1/n)$
spectral gap bound, and~\cite{momentum,Caputo} for generalizations.

If the initial distribution $\mu$ does not have an $\cL^2$ density
with respect to $U_{n-1}$, then direct application of the $\cL^2$
theory above provides no information. The best result for the rate of
convergence when the initial distribution is, say, concentrated at one
point is given in~\cite{PDLSC}, where it was shown that at most
$\mathcal{O}(n^{2n}\log(\varepsilon^{-1}))$ steps are required to get
within $\varepsilon$ close to $U_{n-1}$ in total variation distance. The
$\cL^2$ theory gives a mixing time of $\mathcal{O}(2n \log
(\varepsilon
^{-1})) \|\mu\|_{\cL^2}$.\vadjust{\goodbreak}

If we measure convergence of $K$ or $\tilde{K}$ in terms of other
probability metrics, most notably $\cL^1$ or $\cL^2$ transportation
cost, then the available convergence rate results are much better.
Using comparison techniques, it was shown in~\cite{PDLSC} that
$\mathcal
{O}(n^4 \log n)$ steps suffice for Kac's walk on $SO(n)$ to get
arbitrarily close to stationarity in $\cL^1$ transportation distance,
which metrizes weak convergence. This was improved in~\cite{PS} to an
upper bound of $\mathcal{O}(n^{2.5} \log n)$, using a coupling
argument. Since the standard projection map $\pi\dvtx SO(n) \to S^{n-1}$ can
only decrease Riemannian distance, all the transportation mixing time
results for $SO(n)$ are also valid for $S^{n-1}$. This is of course
true for total variation mixing as well, but unfortunately we cannot
obtain polynomial total variation mixing time for the walk on $SO(n)$.

These suggest that polynomial time mixing should also be true for
total variation distance, since there is nothing pathological about the
walk. The main difficulty in the analysis lies in that the distribution
of the walk at any finite time step will never have a finite $\cL^2$
density with respect to the Lebesgue measure on $S^{n-1}$ if we start
with the point mass. In the following section, however, we will show
that by some removing the singular set of the density after some
burn-in period, and using the fact that total variation distance
between two measures decreases under the evolution of a Markov chain,
one can still essentially use the spectral gap to obtain a~polynomial
bound on the total variation mixing time. More explicitly, we have the following theorem.
\begin{thmm}
Let $K$ denote the Markov kernel for Kac's random walk on the
$(n-1)$-sphere, $S^{n-1} \subset\R^n$, let $U$ denote the uniform
distribution on~$S^{n-1}$, and let~$\delta_{e_1}$ denote the
probability measure concentrated at the point $e_1 = (1,0,\ldots, 0)
\in\R^n$. Then
\[
\|\delta_{e_1} K^t - U\|_{\TV} \le\varepsilon
\]
for $t > c n^5 (\log n)^3 \log\varepsilon^{-1}$, where $c$ is a constant
that does not depend on $n$.
\end{thmm}
\begin{remark*}
1.
For a fixed $\varepsilon$, the proof we give below produces a bound with
an additional factor of $\log\varepsilon^{-2}$ for the mixing time. Now
for general Markov chains on any state space, we have the following
sub-multiplicative property (\cite{LPW}, Section~4.4):
\[
\bar{d}(s+t) \le\bar{d}(s) \bar{d}(s)
\]
for $\bar{d}(s) := \sup_{\mu,\nu} \|\mu K^s - \nu K^s\|_{\TV}$ and
$d(t) :=\sup_\mu\|\mu K^t - \pi\|_{\TV} \le\bar{d}(t) \le2d(s)$. We
deduce that $d(tk) \le(2 d(t))^k$, hence $\tau_{\mix}(\varepsilon)
\le
\log_2 (1/\varepsilon) \tau_{\mix}(1/4)$, that is, the additional factor
can be removed.\vspace*{-6pt}
\begin{longlist}[2.]
\item[2.]
Very recently, I learned that Aaron Smith~\cite{AS} came up with a
coupling argument based on Wasserstein contraction that gets the
correct order $\mathcal{O}(n \log n)$ of total variation mixing time
for the Gibbs sampler on the $n$-simplex. Since Kac's walk on the\vadjust{\goodbreak}
sphere is in fact a Gibbs sampler on the $n$-simplex if one squares the
coordinates, at least if one starts with a measure symmetric under the
transform $\vec{x} \to-\vec{x}$, his argument presumably gives the
same result here as well. But I believe the argument presented here is
of independent interest, especially in comparison analysis, for which
transportation mixing time bound might not be available.

\item[3.] As mentioned above, we are unable to get any polynomial mixing
time result for Kac's walk on $SO(n)$. But in fact, even for the
induced walk on the Grassmanian space, $SO(n)/SO(n-k)$ where $k \ge2$,
polynomial mixing is beyond reach at the moment. The difficulty of
applying the present technique is that the support of the running
distribution cannot be confined into nice submanifolds of the state
space for $k \ge2$, thus an induction based on the dimension of the
support does not work.

\item[4.] Another line of research is concerned with entropy mixing time of
Kac's random walk (see~\cite{Entropy} and references therein). In order
for entropy distance to go down to zero, the starting measure has to
have a density with finite relative entropy with respect to the uniform
measure on $S^{n-1}$. It is not clear whether starting at a point mass
the chain will have finite entropy in finite time.
\end{longlist}
\end{remark*}

\section{Bounding the total variation distance}\label{sec2}

This section gives bounds on the convergence rate of Kac's random walk
on $S^{n-1}$ starting at a standard basis vector~$e_i$, in total
variation distance.

Recall the total variation distance between two probability measures
$\mu$ and $\nu$ on the same probability space $(S, \mathcal{S})$ is
defined by the following variational quantity:
\[
\| \mu- \nu\|_{\TV} = 2\sup_{A \in\mathcal{S}} |\mu(A) -
\nu(A)|,
\]
where $\mathcal{S}$ is the $\sigma$-algebra on $S$.

Alternatively, total variation has the variational characterization in
terms of bounded functions:
\[
\|\mu- \nu\|_{\TV} = \mathop{\sup_{f\dvtx  \|f\|_{\infty} \le1}} |\mu(f)
- \nu(f)|.
\]
This will be used to show the weakly contracting property of Markov
chains under total variation distance below.

Let $A_k$ be the event that at the $k$th step of the walk, every pair
of coordinates has been used. Then we have
\[
P(A_k^c) := \eta_k < \pmatrix{n \cr 2}\biggl(1-\frac{1}{{n \choose2}}\biggr)^k.
\]
Conditioning on this event, we have the following two claims:

\begin{claim}\label{cl1} The density $g :=\frac{d \mu'_k}{d U_{n-1}}$ of the
resulting distribution $\mu'_k$ of the conditioned random walk with\vadjust{\goodbreak}
respect to the uniform distribution on $S^{n-1}$ satisfies the
following bound:
\begin{eqnarray} \label{density bound}
g(x) &\le&\Bigl|\min_{1\le i \le n} x_i\Bigr|^{-n} \Biggl(\sum_{i=1}^n (-\log
|x_i|)^k\Biggr) C^k \prod_{m=1}^k m! \\[-2pt]
&\le& C^k k^{k^2} \Bigl|\min_{1 \le i \le n} x_i\Bigr|^{-n} \Bigl(-\log\Bigl|\min_{1 \le i
\le n} x_i\Bigr|\Bigr)^k =:C(n,k)\vspace*{-1pt}
\end{eqnarray}
for some fixed absolute constant $C$.\vspace*{-1pt}
\end{claim}

\begin{claim}\label{cl2} For $k > - n^2 \log n \log\varepsilon$, and
$\varepsilon<
n^{-3}$, the set $H_{\varepsilon} := \{x \in S^{n-1}\dvtx \break |x_i| <
\varepsilon$
for some $i\}$ satisfies the following bound on its probability under
the $A_k$-conditional distribution:
\begin{equation} \label{epsilon neighborhood bound}
\mu'_k(H_{\varepsilon}) \le\varepsilon^{{1}/{8}}.\vspace*{-1pt}
\end{equation}
\end{claim}

Let us first show how claims 1 and 2 lead to a polynomial time
convergence rate for Kac's walk under total variation norm. Let $\mu_k$
be the distribution on~$S^{n-1}$ after $k$ steps of the random walk,
and let $\mu'_k$ be $\mu_k$ conditional on~$A_k$, that is, for $B
\subset S^{n-1}$,
\[
\mu'_k(B) = P(\delta_{e_1} R^k \in B | A_k),\vspace*{-1pt}
\]
where $R$ is the one-step transition kernel of Kac's random walk.

Then we have
\begin{equation} \label{singular to L^1 TV bound}
\|\mu'_k - \mu_k\|_{\TV} < \eta_k < \pmatrix{n \cr 2}\biggl(1-\frac{1}{{n
\choose2}}\biggr)^k.\vspace*{-1pt}
\end{equation}
To check this, let $B \subset S^{n-1}$ be Lebesgue measurable. Then we have
\begin{eqnarray*}
\mu_k(B) &=& P(\delta_{e_1} R^k \in B|A_k)P(A_k) + P(\delta_{e_1} R^k
\in B|A_k^c)P(A_k^c) \\[-2pt]
& \le&\mu_k'(B) + \eta_k.\vspace*{-1pt}
\end{eqnarray*}
This implies
\[
\mu_k(B) - \mu'_k(B)\le\eta_k.\vspace*{-1pt}
\]
On the other hand, since
\[
\mu'_k(B) =\frac{P(\{ \delta_{e_1} R^k \in B \} \cap A_k)}{P(A_k)},\vspace*{-1pt}
\]
we also get
\[
\frac{\mu_k(B)}{1- \eta_k} > \mu'_k(B)\vspace*{-1pt}
\]
which gives
\[
\mu_k(B) > \mu'_k(B) - \eta_k \mu'_k(B)\vspace*{-1pt}
\]
hence
\[
\mu'_k(B) - \mu_k(B) < \eta_k\vspace*{-1pt}
\]
which establishes \eqref{singular to L^1 TV bound}.\vadjust{\goodbreak}

Next recall that a Markov kernel is weakly contracting in total
variation norm because if $f$ is a bounded continuous function on the
state space with
\[
\|f\|_{\infty} \le1,\vspace*{-1pt}
\]
then $Rf(x) = \int R(x, dy) f(y)$ satisfies the same $\cL^{\infty}$
bound, hence
\[
(\mu R - \nu R)(f) = (\mu- \nu)(Rf) \le\|\mu- \nu\|_{\TV}.\vspace*{-1pt}
\]

Thus by the triangle inequality we just need to bound $\|\mu'_k R^l -
U_{n-1}\|_{\TV}$ from now on, where $U_{n-1}$ denotes the uniform
distribution on $S^{n-1}$, and at the end add~$\eta_k$ to the
resulting bound.

Next we modify $\mu'_k$ to a different distribution $\nu_k$ as
follows. We define $\nu_k$ in terms of its density with respect to $U_{n-1}$.

On the set $H_{\varepsilon}^c$,
\[
\frac{d\nu_k}{dU_{n-1}} := \frac{d\mu'_k}{dU_{n-1}}.\vspace*{-1pt}
\]

On the set $H_{\varepsilon}$, we let its density be a constant equal to
the mass of $H_{\varepsilon}$ under $\mu'_k$ divided by its mass under
$U_{n-1}$, which is what's needed for $\nu_k$ to be a probability
distribution on $S^{n-1}$; we invoke Claim~\ref{cl1} above to get an upper
bound on this constant:
\begin{eqnarray*}
\frac{d\nu_k}{dU_{n-1}} &\equiv&\frac{\mu'_k(H_{\varepsilon
})}{U_{n-1}(H_{\varepsilon})} \\[-2pt]
&<& \frac{\varepsilon^{1/4}}{\varepsilon({\Gamma(n/2)})/({\Gamma
((n-1)/2)\Gamma(1/2)})} \\[-2pt]
&<& \frac{\varepsilon^{1/4}}{\varepsilon\sqrt{({n-2})/{2\pi}}}\\[-2pt]
&<& \varepsilon^{-{3}/{4}} \sqrt{\frac{2\pi}{n-2}}.\vspace*{-1pt}
\end{eqnarray*}
In the computation above we used two ingredients. First we used that
\begin{equation} \label{ingredient 1}
\frac{\Gamma(n/2)}{\Gamma((n-1)/2)} > \sqrt{\frac{n-2}{2}}\vspace*{-1pt}
\end{equation}
which follows from log convexity of the $\Gamma$ function. Since
$\frac
{1}{2} (\log\Gamma(n) + \log\Gamma(n-1)) > \log\Gamma(n-1/2)$, we get
\[
\frac{\Gamma(n)}{\Gamma(n-1/2)} > \frac{\Gamma(n-1/2)}{\Gamma(n-1)},\vspace*{-1pt}
\]
which implies \eqref{ingredient 1} above. By incrementing $n$ by $1/2$,
we also get a reverse inequality of the form
\begin{equation} \label{ingredient 1.5}
\frac{\Gamma(n/2)}{\Gamma((n-1)/2)} < \sqrt{\frac{n-1}{2}}.\vspace*{-1pt}
\end{equation}
This will be useful later when we bound $U(H_{\varepsilon})$ in the
proof\vadjust{\goodbreak}
of Claim~\ref{cl2}.

The second ingredient is the formula for the coordinate marginal
density for the uniform distribution on the sphere (see~\cite{PDLSC}
but with a small typo, namely by $n$-sphere they meant $(n-1)$-sphere):
\begin{equation} \label{sphere margin}
\frac{d}{da}\PP_U(x_1 \in[-1,a]) = \frac{\Gamma((n+1)/2)}{\Gamma(1/2)
\Gamma(n/2)}(1-a^2)^{(n-2)/2},
\end{equation}
where $P_U$ denotes uniform distribution on $S^{n-1}$.

The total variation distance between $\mu'_k$ and $\nu_k$ is given
simply by their total variation distance over the region
$H_{\varepsilon
}$, hence we have
\begin{eqnarray} \label{L^1 to L^2 TV bound}
\|\mu'_k - \nu_k\|_{\TV} &\le&\mu'_k(H_{\varepsilon}) +
\frac{n
\Gamma({n}/{2})}{\Gamma({1}/{2}) \Gamma({(n-1)}/{2})}
\varepsilon\\
&\le& n^{3/2} \varepsilon+ \varepsilon^{1/8}.
\end{eqnarray}

Thus by choosing $\varepsilon$ sufficiently small, whose exact value we
will determine in the end, we can make sure that $\mu'_k$ and $\nu_k$
are very close in total variation distance. And again by weak
contractivity of Markov kernel, we now simply need to focus on bounding
$\|\nu_k R^l - U_{n-1}\|_{\TV}$. Since $\nu_k$ has an $\cL^2$
density with respect to $U_{n-1}$, we can use the spectral gap to bound
the rate of convergence. First we bound the $\cL^2(dU_{n-1})$ distance
between $\nu_k$ and $U_{n-1}$:
\begin{eqnarray} \label{L^2 distance}
&&\|\nu_k - U_{n-1}\|_{\cL^2(d U_{n-1})}
\nonumber
\\[-8pt]
\\[-8pt]
\nonumber
&&\qquad= \biggl(\int_{H_{\varepsilon}}
\biggl|\frac{d
\nu_k}{dU_{n-1}} - 1\biggr|^2 \,dU_{n-1} + \int_{H_{\varepsilon}^c}\biggl |\frac
{d \nu
_k}{dU_{n-1}} -1\biggr|^2 \,dU_{n-1}\biggr)^{{1}/{2}}.
\end{eqnarray}
Let us bound the two integrals separately.

For the first integral on the right-hand side of \eqref{L^2 distance},
we have
\begin{eqnarray}\label{bound of first integral}
\int_{H_{\varepsilon}} \biggl|\frac{d \nu_k}{dU_{n-1}} - 1\biggr|^2 \,dU_{n-1}
&\le&
\int_{H_{\varepsilon}} \biggl(\frac{d\nu_k}{dU_{n-1}}\biggr)^2 \,dU_{n-1} +
U_{n-1}(H_{\varepsilon}) \nonumber\\
&<& \varepsilon^{-3/2} \frac{8\pi}{n-2} \varepsilon\frac{\Gamma
({n}/{2})}{\Gamma({(n-1)}/{2})\Gamma({1}/{2})} \\
&<& 4 \varepsilon^{-1/2} \sqrt{\frac{2\pi}{n-2}}.\nonumber
\end{eqnarray}

For the second integral, notice first that $H_{\varepsilon}^c$ is the set
of points on $S^{n-1}$ for which all the coordinates are greater than
$\varepsilon$. So Claim~\ref{cl2} tells us that the density $\frac{d\nu
_k}{dU_{n-1}}$ over this region is bounded above by $\varepsilon^{-n}$,
from which we immediately get the following bound:
\begin{equation} \label{bound of second integral}
\int_{H_{\varepsilon}^c} \biggl|\frac{d \nu_k}{dU_{n-1}} - 1\biggr|^2 \,dU_{n-1} <
\varepsilon^{-2n} + 1.
\end{equation}

Combining \eqref{bound of first integral} and \eqref{bound of second
integral}, we get, for $\varepsilon< \frac{1}{2}$ and $n > 2$, say, that
\[
\|\nu_k - U_{n-1}\|_{\cL^2(dU_{n-1})} \le2 \varepsilon^{-n}.\vadjust{\goodbreak}
\]

By the results in~\cite{Kac}, we know that the spectral gap of the Kac
kernel is $\frac{1}{2n}$, so we get
\begin{eqnarray} \label{spectral bound}
\|\nu_k R^l - U_{n-1}\|_{\TV} &\le&\biggl\|\frac{d\nu_k}{dU_{n-1}}
-1\biggr\|
_{\cL^2(dU_{n-1})} \biggl(1-\frac{1}{2n}\biggr)^m
\nonumber
\\[-9pt]
\\[-9pt]
\nonumber
&\le&2 \varepsilon^{-n} \biggl(1-\frac{1}{2n}\biggr)^m.\vspace*{-1pt}
\end{eqnarray}

Finally, combining \eqref{singular to L^1 TV bound} \eqref{L^1 to L^2
TV bound} and \eqref{spectral bound}, we get
\begin{eqnarray}\label{final TV bound}
\| \delta_{e_1} R^{k+l} - U_{n-1}\|_{\TV} &\le&\pmatrix{n \cr 2}
\biggl(1-\frac{1}{{n \choose2}}\biggr)^k + n^{3/2} \varepsilon+ \varepsilon
^{1/8}
\nonumber
\\[-9pt]
\\[-9pt]
\nonumber
&&{}+ C^k
k^{k^2} |\varepsilon|^{-n} (-\log\varepsilon)^k \biggl(1-\frac{1}{2n}\biggr)^l.\vspace*{-1pt}
\end{eqnarray}

So it remains to minimize the right-hand side of \eqref{final TV
bound} with respect to $k$ and $l$.

Suppose our target total variation distance is $3\delta$. Then we can
simply divide~$3\delta$ into three equal parts and bound each summand
in \eqref{final TV bound} by $\delta$. We look at each summand below:

Bounding the first summand yields
\[
\pmatrix{n \cr 2} \biggl(1-\frac{1}{{n \choose2}}\biggr)^k < \delta\quad
\Rightarrow\quad k > (-\log\delta+ 2\log n )\pmatrix{n \cr 2}.\vspace*{-1pt}
\]

So it suffices to take
\begin{equation} \label{first summand}
k > n^2 \log n \log\frac{1}{\delta}.\vspace*{-1pt}
\end{equation}

Bounding the second summand $\varepsilon^{1/8} + n^{3/2} \varepsilon<
\delta
$, it suffices to have
$\varepsilon^{1/8} < \delta/2$ and $n^{3/2} \varepsilon< \delta/2$,
which gives
\[
\varepsilon< \tfrac{1}{2} \delta^8 n^{-3/2}.\vspace*{-1pt}
\]
But taking $\varepsilon= n^{-3} \delta^8$ certainly fulfills that, which
will affect the bound on $l$ in the third summand:
\[
 C^k k^{k^2} |\varepsilon|^{-n} (-\log\varepsilon)^k \biggl(1-\frac
{1}{n}\biggr)^l < \delta\vspace*{-1pt}
\]
implies we need $l$ greater than
\begin{eqnarray*}
&& 2n\bigl(-\log\delta+ k \log C + k^2 \log k - n \log\varepsilon+ k \log
\log(\varepsilon^{-1}) \bigr) \\[-2pt]
&&\qquad< n\bigl(-\log\delta+ k \log C \\[-2pt]
&&\qquad\hspace*{18pt}{}+ n^4 (\log n)^2 (\log\delta)^2 \bigl(2 \log n
+ \log\log n + \log\log(\delta^{-1})\bigr)\\[-2pt]
&&\hspace*{78pt}\qquad\quad{}+n(-8 \log\delta+ 3 \log n) + k \log\log\varepsilon^{-1}\bigr)\\[-2pt]
&&\qquad< C' n^5 (\log n)^3 (\log\delta)^3\vspace*{-1pt}
\end{eqnarray*}
for some constant $C'$.\vadjust{\goodbreak}

Clearly $l$ dominates $k$, so it requires a total of $C' n^5 (\log n)^3
(\log\delta)^3$ steps to bring the running distribution of Kac's
random walk to be $3\delta$ close to its stationary distribution on the
unit sphere $S^{n-1}$.

Finally we prove the two claims introduced in the beginning.

\section{\texorpdfstring{Proof of Claim \protect\ref{cl1}}{Proof of Claim 1}}\label{sec3}
Starting at the delta mass at $e_1$, an admissible sequence of
rotations in $A_k$ will distribute it over the entire $S^{n-1}$ with
positive probability everywhere provided that $P(A_k) > 0$, that is,
for sufficiently large $k$. This will certainly be the case if $k \ge-
n^2 \log n \log\delta$ for $- \log\delta> 2$. So we will look at the
conditional probability density given that the walk has taken a
sequence of steps in $A_k$, and we will estimate the density growth
from step $j-1$ to $j$, up to step $k$.

Observe that at step $j-1$, $j \le k$, the support of the running
distribution is a subsphere of $S^{n-1}$. Without loss of generality,
we call this subsphere~$S^m$. Denote by $u_j, v_j$ the axes that span
the plane along which the rotation~$\gamma_j$ takes place.

The way $\gamma_j$ affects the previous running distribution can be
classified into three cases:

1. $u_j, v_j \notin S^m$, in which case the running distribution
remains unchanged.

2. $u_j, v_j \in S^m$, in which case the support after the rotation is
still on $S^m$, but the density might change.

3. $u_j \in S_m$, $v_j \notin S_m$, in which case the support of the
running distribution grows to be a sphere with one dimension higher
than $S^m$, denoted without loss of generality $S^{m+1}$.

Case 1 clearly does not increase the density of the running
distribution, because the rotation does not take $S^m$ outside itself
and for $\theta\in[0, 2\pi]$, the density at $(x_1, \ldots, x_{m+1},
\ldots, (u_j^2 + v_j^2)^{1/2} \cos\theta, \ldots, (u_j^2+v_j^2)^{1/2}
\sin\theta, \ldots, x_n)$ with respect to $U_m$ only depends on the
first $m+1$ coordinates, which means that averaging over $\theta$
uniformly in $[0, 2\pi]$ remains the same.

To understand Case 3, first observe that there can be at most $n$ such
steps in the history of the Kac walk. So if we can show each type 3
rotation increases the density by at most $|\min_{1 \le i \le n}
x_i|^{-1}$, then the factor $|\min_{1 \le i \le n} x_i|^{-n}$ would be
taken care of. This is the content of the following lemma.

\begin{lem}
Assuming the running density $h_m(x_1, \ldots, x_{m+1})$ with respect
to $U_m$ after step $j-1$ is bounded by $g_m(x_1, \ldots x_{m+1})$, and
that without loss of generality $u_j = x_{m+1}$, $v_j = x_{m+2}$, then
the new density $h_{m+1}(x_1, \ldots,\break x_{m+2})$ with respect to
$U_{m+1}$ after step $j$ is bounded by
\[
\frac{1}{2 \pi} g_m\bigl(x_1, \ldots, (x_{m+1} + x_{m+2})^{1/2}\bigr)
(x_{m+1}^2+x_{m+2}^2)^{-1/2}.
\]
\end{lem}

Observe that $(x_{m+1}^2 + x_{m+2}^2)^{-1/2} \le|\min_{1 \le i \le n}
x_i|^{-1}$.\vadjust{\goodbreak}

\begin{pf}
Denote the new density with respect to $U_{m+1}$ by $h_{m+1}(x_1,
\ldots,\break x_{m+2})$ with a slight abuse of notation. Then we have
\[
h_{m+1}\bigl(x_1, \ldots, (x_{m+1}^2 + x_{m+2}^2)^{1/2} \cos\theta,
(x_{m+1}^2 + x_{m+2}^2)^{1/2} \sin\theta\bigr)
\]
is independent of $\theta$ and in particular equals
\[
h_{m+1}\bigl(x_1, \ldots, (x_{m+1}^2 + x_{m+2}^2)^{1/2}, 0\bigr).
\]
Furthermore, the total contribution of density from $(x_1, \ldots,
(x_{m+1}^2 +\break x_{m+2}^2)^{1/2} \times \cos\theta, (x_{m+1}^2 +
x_{m+2}^2)^{1/2} \sin\theta)$ for all $\theta$ should add up to the
previous density at the point $(x_1, \ldots, (x_{m+1}^2 +
x_{m+2}^2)^{1/2})$. In other words,
\begin{eqnarray*}
&&(x_{m+1}^2 + x_{m+2}^2)^{1/2}\\
&&\quad{}\times \int_{\theta=0}^{2\pi} h_{m+1}\bigl(x_1,
\ldots, (x_{m+1}^2 + x_{m+2}^2)^{1/2} \cos\theta, (x_{m+1}^2 +
x_{m+2}^2)^{1/2} \sin\theta\bigr) \,d\theta\\
&&\qquad= h_m\bigl(x_1, \ldots, \ldots, (x_{m+1}^2 + x_{m+2}^2)^{1/2}\bigr).
\end{eqnarray*}
Notice that the factor $(x_{m+1}^2 + x_{m+2}^2)^{1/2}$ accounts for the
measure of the circle $\{(y_1, \ldots, y_{m+2}, 0, \ldots, 0)$: with
$y_1 = x_1, \ldots,y_m = x_m$ and $y_{m+1}^2 + y_{m+2}^2 =
x_{m+1}^2 + x_{m+2}^2 \}$, over which we aggregate.

Thus we get
\begin{eqnarray*}
&& h_{m+1}\bigl(x_1, \ldots, (x_{m+1}^2 + x_{m+2}^2)^{1/2} \cos\theta,
(x_{m+1}^2 + x_{m+2}^2)^{1/2} \sin\theta\bigr)\\
&&\qquad= \frac{1}{2\pi} (x_{m+1}^2 + x_{m+2}^2)^{-1/2} h_m\bigl(x_1, \ldots,
(x_{m+1}^2 + x_{m+2}^2)^{1/2}\bigr)\\
&&\qquad\le\frac{1}{2\pi} (x_{m+1}^2 + x_{m+2}^2)^{-1/2} g_m\bigl(x_1, \ldots,
(x_{m+1}^2 + x_{m+2}^2)^{1/2}\bigr).
\end{eqnarray*}
\upqed\end{pf}

The Case 2 rotations will contribute the remaining factors in the
bound of $g(x)$ in Claim~\ref{cl1}. More precisely, we have the following lemma.

\begin{lem}
Assume at step $j-1$, the running distribution is supported on some
$S^m \subset S^{n-1}$, which is viewed as the standard sphere in $\R
^{m+1} = \{x_1, \ldots, x_{m+1}\}$, and that the density $g_j$ with
respect to $U_m$ satisfies
\begin{eqnarray} \label{inductive bound}
&&g_{j-1}(x_1, \ldots, x_{m+1}) \nonumber\\
&&\qquad\le C(j,m) (a_1^2 + b_1^2)^{-1/2}
\cdots(a_{m-1}^2 + b_{m-1}^2)^{-1/2}\\
&&\quad\qquad{}\times  [(-\log|x_1|)^{j-1} + \cdots+
(-\log|x_{m+1}|)^{j-1}],\nonumber
\end{eqnarray}
where $C(j,m)$ is a constant that varies with $j$ and $m$. Here $a_i
\neq b_i$ for each $i$ and $(a_1, b_1), \ldots, (a_{m-1}, b_{m-1})$ are
pairs in $\{x_1, \ldots, x_{m+1}\}^2$ with the property that no two
pairs are the same and each coordinate appears at most twice.

If furthermore the $j$th rotation is as in Case 2, then the new
density bound takes the form
\begin{eqnarray*}
&& g_j(x_1, \ldots, x_{m+1}) \\
&&\qquad\le512 C(j,m) (j+1)! (a_1^2 +
b_1^2)^{-1/2} \cdots(a_{m-1}^2 + b_{m-1}^2)^{-1/2}\\
&&\quad\qquad{}\times [(-\log|x_1|)^j + \cdots+ (-\log|x_{m+1}|)^j]
\end{eqnarray*}
with possibly a different sequence of $(a_i, b_i)$ satisfying the same
property as above.
\end{lem}

Notice that starting with a density satisfying the bound \eqref
{inductive bound}, a type 1 or type 3 rotation would preserve its form,
with $j$ replaced by $j+1$. Type 1 rotation does that trivially, due to
the fact that the polylogarithmic factor always increases with~$j$.
Type 2 rotation introduces an additional factor of $(a_m^2 +
b_m^2)^{-1/2}$, but decreases the other existing factors, hence also
preserves the bound with $j \to j+1$.

\begin{pf}
Without loss of generality assume $(u_j,v_j)=(1,2)$.

The new density $h'$ is obtained from the old density $h$ by averaging
over $\theta\in[0, 2\pi]$ of $h(R(1,2,\theta) x)$, where
$R(1,2,\theta
)x$ denotes the rotation of the vector $x \in S^m$ by angle $\theta$
along $x_1 \wedge x_2$. In formula, we have
\begin{equation} \label{circle averaging}
h'(x) = \frac{1}{2\pi} \int_0^{2\pi} h(R(1,2,\theta) x) \,d\theta.
\end{equation}
We write the bound \eqref{inductive bound} as a sum of $m+1$ terms and
consider one of the terms
\[
g_i(x) = C (a_1^2 + b_1^2)^{-1/2} \cdots(a_{m-1}^2 +
b_{m-1}^2)^{-1/2} (-\log|x_i|)^{j-1}.
\]
By assumption, at most two elements in $a_1, b_1, \ldots, a_{m-1}, b_{m-1}$
equal $x_1$ and at most two other elements equal~$x_2$.

By the circle averaging formula \eqref{circle averaging}, we have
\[
g_i'(x) = \frac{1}{2\pi} \int_0^{2\pi} g\bigl((x_1^2 + x_2^2)^{1/2}\cos
\theta,(x_1^2 + x_2^2)^{1/2}\sin\theta,x_3, \ldots, x_{m+1}\bigr) \,d\theta.
\]
We shall break the integral into two parts, where the range of
integration is over $I_{\cos} = [0,\pi/4] \cup[3\pi/4,5\pi/4] \cup
[7\pi/4,2\pi]$ and its complement $I_{\sin}$ in $[0,2\pi]$,
respectively; that is, the ranges are where $\cos\theta$ is close to~$1$ and $\sin\theta$ is close to~$1$, respectively. By symmetry, we
just have to deal with the integral over the range $\theta\in I_{\sin
}$, and multiply the final bound by $1$ in the end.

First we look at the case when $i \notin\{1,2\}$, which means the
rotation $(1,2)$ does not affect the logarithmic factor $(-\log
|x_i|)^j$ at the end. In this case, all the factors in $g_i(x)$ of the
form $(x_2^2 + x_s^2)^{-1/2}$ that involve $x_2$ but not $x_1$ upon the
rotation $R(1,2,\theta)$ become $((x_1^2 + x_2^2) \sin^2 \theta+
x_s^2)^{-1/2}$, which can be bounded above by $\sqrt{2} (x_1^2 + x_2^2
+ x_s^2)^{-1/2}$.

As of the factors that involve both $x_1$ and $x_2$, that is, $(x_1^2
+ x_2^2)^{-1/2}$, there can be at most one of such. And it remains the
same under the rotation $R(1,2,\theta)$ since $(x_1^2 + x_2^2)\cos^2
\theta+ (x_1^2 + x_2^2) \sin^2 \theta= x_1^2 + x_2^2$.

The factors that involve $x_s$ and $x_s$, $s \neq2$, become $((x_1^2
+ x_2^2) \cos^2 \theta+ x_s^2)^{-1/2}$, which we can bound as follows:

Using the fact that $\frac{1}{\sqrt{2}} (|a| + |b|) \le(a^2 +
b^2)^{1/2} \le|a| + |b|$, we get
\[
\bigl((x_1^2 + x_2^2) \cos^2 \theta+ x_s^2\bigr)^{-1/2} \sim[(|x_1| + |x_2|)
|\cos\theta| + |x_s|]^{-1},
\]
where $a \sim b$ means $b/C \le a \le bC$ for some constant $C$. Here
we can take $C$ to be~$2$.

More difficult is the case when $i \in\{1,2\}$, when we also have to
deal with a $(-\log[(x_1^2 + x_2^2)^{-1/2} \cos\theta])^{j-1}$ factor
that goes to infinity for $\theta\in I_{\sin}$.

In fact when $i =1$, the only factors that have singularities for
$\theta\in I_{\sin}$ and for the coordinates bounded away from $0$
take the following form:
\begin{eqnarray*}
&&\bigl((|x_1| + |x_2|) |\cos\theta| + |x_s|\bigr)^{-1}\bigl((|x_1| + |x_2|) |\cos
\theta| + |x_t|\bigr)^{-1} \\
&&\qquad{}\times \bigl(-\log[(x_1^2 + x_2^2)^{-1/2} \cos\theta]\bigr)^j,
\end{eqnarray*}
where $s\neq t$, or without the $x_t$ factor. In the former case we
will show in Lemma~\ref{factorial lemma} below that the following integral:
\begin{eqnarray*}
&&\frac{1}{2\pi} \int_{\theta\in I_{\sin}} \bigl((|x_1| + |x_2|) |\cos
\theta| + |x_s|\bigr)^{-1}\bigl((|x_1| + |x_2|) |\cos\theta| + |x_t|\bigr)^{-1}\\
&&\hspace*{21pt}\qquad{}\times \bigl(-\log[(x_1^2 + x_2^2)^{1/2} \cos\theta]\bigr)^{j-1} \,d\theta
\end{eqnarray*}
is bounded by
\begin{eqnarray} \label{two inverse singular terms}
&&j! (x_1^2 + x_2^2)^{-1/2} (x_s^2 + x_t^2)^{-1/2}
\nonumber
\\[-8pt]
\\[-8pt]
\nonumber
&&\qquad{}\times[(-\log|x_1|)^j +
(-\log|x_2|)^j + (-\log|x_s|)^j + (-\log|x_t|)^j]
\end{eqnarray}
whereas in the case where the $x_t$ factor is not present, the same
bound \eqref{two inverse singular terms} multiplied by $\sqrt{2}$
applies the expression
\begin{equation} \label{only one inverse singular term}
\qquad\frac{1}{2\pi} \int_{\theta\in I_{\sin}} \bigl((|x_1| + |x_2|) |\cos
\theta
| + |x_s|\bigr)^{-1}\bigl(-\log[(x_1^2 + x_2^2)^{1/2} \cos\theta]\bigr)^{j-1}
\,d\theta
\end{equation}
using the fact that for $\theta\in I_{\sin}$,
\[
\bigl((|x_1| + |x_2|) \cos\theta+ |x_t|\bigr)^{-1} \ge1/\sqrt{2}.
\]

When $i \neq1$, the logarithmic singularity will not arise when
integrating over $\theta\in I_{\sin}$, so it will trail off as a
remaining factor of the form $(-\log|x_i|)^{j-1} \le1 + (-\log|x_i|)^j$.

Recall also that we have factors of the form
\begin{equation} \label{uniform bounds from nonsingular factors}
2 (x_1^2 + x_2^2 + x_s^2)^{-1/2} (x_1^2 + x_2^2 + x_t^2)^{-1/2}
\end{equation}
coming from the uniform bound on the factors involving $x_2$ but not
$x_1$; here~$s,t$ are possibly different indices than those appearing
in the singular factors. Equation~\eqref{uniform bounds from
nonsingular factors} can be trivially bounded above by $ 2(x_1^2 +
x_s^2)^{-1/2}(x_2^2 + x_t^2)^{-1/2}$.
The remaining inverse factors in $g(R(1,2,\theta)x)$ do not
contain~$x_1$ or~$x_2$, so one can easily check that the inductive hypothesis
is satisfied.

The best way to visualize this branching inductive argument is to
consider a simple, possibly disconnected graph on $m+1$ vertices with
degrees bounded above by 2. The edges between $i$ and $j$ represent a
factor of the form $(x_i^2 + x_j^2)^{-1/2}$ in the bound on the
density. A rotation in the $x_1 \wedge x_2$ plane has the effect of
producing two new graphs on $m+1$ vertices, and the density bound we
get will be a sum over all the resulting graphs. Without loss of
generality let us describe one of those two descendant graphs, the one
associated with $x_1$.

There will be edges $(1,2)$, $(3,4)$, $(1,3)$ and $(2,4)$ if $x_3$ and
$x_4$ were incident to $x_1$ in the previous graph, or simply $(1,2)$
when $x_1$ only has degree $1$. If $x_1$ had degree $0$, then it
remains isolated in the $x_1$ component of the descendant graph. In the
process of this rewiring, some logarithmic factors $(\log|x_s|)^j$ and
factorial factors $j!$ are also introduced, namely, if $(-\log
|x_3|)^{j-1}$ or $(-\log|x_4|)^{j-1}$ was a factor in the bound for the
previous step running disribution, then the new bound will have
$j!(-\log|x_4|)^j$. If there is originally a log factor of other
coordinates, then the exponent on that factor remains the same.

It remains to prove the bound \eqref{two inverse singular terms}, and
notice that we only need to prove it for $\theta\in[\pi/4, \pi/2]$
and then multiply the resulting bound by $4$. This is given by the
following technical lemma.
\end{pf}

\begin{lem} \label{factorial lemma}
For $0 \le x_t, x_s$, $0 \le x_1, x_2$,
\begin{eqnarray*}
&&\int_0^{\pi/4} \bigl((x_1 + x_2) \sin\theta+ x_s\bigr)^{-1} \bigl((x_1 + x_2)
\sin
\theta+ x_t\bigr)^{-1}\\
&&\hspace*{4pt}\qquad{}\times  \bigl(-\log[(x_1 + x_2) \sin\theta]\bigr)^{j-1} \,d\theta\\
&&\qquad\le4 (j+1)! (x_s + x_t)^{-1} (x_1 + x_2)^{-1}\\
&&\hspace*{8pt}\qquad{}\times  [(-\log x_1 )^j +
(-\log x_2)^j + (-\log x_s)^j + (-\log x_t)^j].
\end{eqnarray*}
\end{lem}
\begin{remark*} Note this is equivalent to the bound \eqref{two inverse
singular terms}, by replacing $\sin$ with $\cos$ and changing the range
of integration to $[\pi/4,\pi/2]$.
\end{remark*}
\begin{pf}
Without loss of generality, we can assume $x_t \le x_s$. Furthermore,
we can replace $\sin\theta$ by its linearization at $0$, and multiply
the resulting bound by $2$ in the end, since for $\theta\in[0, \pi
/4]$, we have $\theta/2 \le\sin\theta\le2 \theta$. So instead we
just need to prove
\begin{eqnarray*}
&&\int_0^1 \bigl((x_1 + x_2) \theta+ x_s\bigr)^{-1} \bigl((x_1 + x_2) \theta+
x_t\bigr)^{-1} \bigl(-\log[(x_1 + x_2) \theta]\bigr)^{j-1} \,d\theta\\
&&\qquad\le4 \pi(j+1)! (x_s + x_t)^{-1} (x_1 + x_2)^{-1}\\
&&\hspace*{8pt}\qquad{}\times [(-\log x_1 )^j +
(-\log x_2)^j + (-\log x_s)^j + (-\log x_t)^j].
\end{eqnarray*}

First of all, the factor $((x_1 + x_2) \theta+ x_s)^{-1}$ can be
bounded above by $2(x_s + x_t)^{-1}$ for $\theta\in[0, 1]$.
So it remains to bound the integral of the remaining factors:
\begin{eqnarray*}
&&\int_0^1\bigl ((x_1 + x_2) \theta+ x_t\bigr)^{-1} \bigl(-\log[(x_1 + x_2) \theta
]\bigr)^{j-1} \,d\theta\\
&&\qquad\le x_t ^{-1} \int_0^{\varepsilon} \bigl(-\log[(x_1+x_2) \theta]\bigr)^{j-1}
\,d\theta\\
&&\qquad\hspace*{8pt}{}+ \bigl(-\log[(x_1 + x_2) \varepsilon]\bigr)^{j-1} \int_{\varepsilon
}^1 \bigl((x_1 +
x_2) \theta+ x_t\bigr)^{-1} \,d\theta\\
&&\qquad\le x_t^{-1} j! \varepsilon\bigl(-\log[(x_1 + x_2) \varepsilon]\bigr)^j \\
&&\qquad\hspace*{8pt}{}+
\bigl(-\log
[(x_1 + x_2) \varepsilon]\bigr)^{j-1} (x_1 + x_2)^{-1} \log\biggl[\frac{x_1 +
x_2 +
x_t}{(x_1 + x_2) \varepsilon+ x_t}\biggr] \\
&&\qquad\le x_t^{-1} j! \varepsilon\bigl(-\log[(x_1 + x_2) \varepsilon]\bigr)^j \\
&&\qquad\hspace*{8pt}{}+
\bigl(-\log
[(x_1 + x_2) \varepsilon]\bigr)^{j-1} (x_1 + x_2)^{-1} \log[(x_1 + x_2)
\varepsilon]\\
&&\qquad= x_t^{-1} j! \varepsilon\bigl(-\log[(x_1 + x_2) \varepsilon]\bigr)^j +
\bigl(-\log
[(x_1 + x_2) \varepsilon]\bigr)^j (x_1 + x_2)^{-1}.
\end{eqnarray*}
In the second equality we used the fact that
\[
\int_0^{\varepsilon} (- \log\theta)^j \,d\theta= \int_{-\log
\varepsilon
}^{\infty} y^j e^{-y} \,dy \le j!
\]
and in the third inequality we used $\frac{\varepsilon(x_1+x_2) + x_t}{x_1
+ x_2 + x_t} > \varepsilon(x_1 + x_2)$ for $\varepsilon< 1$.\vspace*{1pt}

Taking $\varepsilon= x_t$, we obtain the result.
\end{pf}

\section{\texorpdfstring{Proof of Claim \protect\ref{cl2}}{Proof of Claim 2}}\label{sec4}
We prove the claim by a contradiction argument. Here we use the result
from~\cite{Oliviera} that after $k = n^2 \log n \log\varepsilon$ steps
the $\cL^2$ transportation distance between the running distribution of
the Kac random walk on $S^{n-1}$ and the uniform distribution $U_{n-1}$
is less than $\varepsilon$. So by Holder's inequality, the~$\cL^1$
transportation distance is also less than $ \varepsilon$. We know that the
uniform measure $U_{n-1}(H_{\varepsilon})$ varies linearly with
$\varepsilon
$; in fact using the marginal density formula \eqref{sphere margin} for
a single coordinate on the unit sphere, together with the fact that $H
= \bigcup H^i_{\varepsilon}$ where $H^i_{\varepsilon} := \{x \in S^{n-1}\dvtx
|x_i| \le\varepsilon\}$, one sees that it is bounded above by $n^{3/2}
\varepsilon$, and similarly $U_{n-1}(H_{\varepsilon^{\alpha} +
\varepsilon}) \le
2 n^{3/2} (\varepsilon^{\alpha} + \varepsilon)$.
Next let $\alpha,\beta$ be two real numbers between $0$ and $1$ satisfying
\[
\alpha+ \beta< 1
\]
and
\[
\alpha- \beta> 1/2.
\]
Then with $\varepsilon\le n^{-3}$, one verifies easily that
\begin{equation} \label{transport cost}
\bigl(\varepsilon^{\beta} - (\varepsilon^{\alpha} + \varepsilon)
n^{3/2}\bigr) \varepsilon
^{\alpha} > \varepsilon.
\end{equation}
So if $\mu_k(H_{\varepsilon}) > \varepsilon^{\beta}$, with
$\varepsilon\le
n^{-3}$, then in order to transport the mass of~$H_{\varepsilon}$ under~$\mu_k$
in excess of $H_{\varepsilon+ \varepsilon^{\alpha}}$ under
$U_{n-1}$, the left-hand side of \eqref{transport cost} gives a~lower
bound on the transportation cost for that alone, because each particle
of mass originally in $H_{\varepsilon}$ must traverse at least a distance
of $\varepsilon^{\alpha}$ to go outside of $H_{\varepsilon+
\varepsilon^{\alpha
}}$. Since the \textit{total} transport cost cannot exceed
$\varepsilon$
after $k$ steps, this is a contradiction. Hence we must have $\mu
_k(H_{\varepsilon}) < \varepsilon^{\beta}$. One set of choices for~$\alpha$
and $\beta$ is $\alpha= 3/4$ and $\beta= 1/8$, which is the content of
Claim~\ref{cl2}.

\section*{Acknowledgments}
I would like to thank my advisor Persi Diaconis for introducing me to
this problem and pointing me to the relevant literature. I would also
like to thank the referee for suggesting numerous changes.

%


\printaddresses


\begin{thebibliography}{12}

\bibitem{Caputo}
\begin{barticle}[mr]
\bauthor{\bsnm{Caputo},~\bfnm{Pietro}\binits{P.}}
(\byear{2008}).
\btitle{On the spectral gap of the {K}ac walk and other binary collision
  processes}.
\bjournal{ALEA Lat. Am. J. Probab. Math. Stat.}
\bvolume{4}
\bpages{205--222}.
\bid{issn={1980-0436}, mr={2429910}}
\bptok{imsref}%
\end{barticle}
\endbibitem

\bibitem{Entropy}
\begin{barticle}[mr]
\bauthor{\bsnm{Carlen},~\bfnm{Eric~A.}\binits{E.~A.}},
  \bauthor{\bsnm{Carvalho},~\bfnm{Maria~C.}\binits{M.~C.}},
  \bauthor{\bsnm{Le~Roux},~\bfnm{Jonathan}\binits{J.}},
  \bauthor{\bsnm{Loss},~\bfnm{Michael}\binits{M.}} \AND
  \bauthor{\bsnm{Villani},~\bfnm{C{\'e}dric}\binits{C.}}
(\byear{2010}).
\btitle{Entropy and chaos in the {K}ac model}.
\bjournal{Kinet. Relat. Models}
\bvolume{3}
\bpages{85--122}.
\bid{doi={10.3934/krm.2010.3.85}, issn={1937-5093}, mr={2580955}}
\bptnote{check year}%
\bptok{imsref}%
\end{barticle}
\endbibitem

\bibitem{Kac}
\begin{barticle}[auto:STB|2012/03/21|07:41:58]
\bauthor{\bsnm{Carlen},~\bfnm{Eric~A.}\binits{E.~A.}},
  \bauthor{\bsnm{Carvalho},~\bfnm{M.~C.}\binits{M.~C.}} \AND
  \bauthor{\bsnm{Loss},~\bfnm{Michael}\binits{M.}}
(\byear{2003}).
\btitle{Determination of the spectral gap for Kac's master equation and related
  stochastic evolutions}.
\bjournal{Acta Math.}
\bvolume{191}
\bpages{1--54}.
\bid{mr={2020418}}
\bptok{imsref}%
\end{barticle}
\endbibitem

\bibitem{momentum}
\begin{barticle}[auto:STB|2012/03/21|07:41:58]
\bauthor{\bsnm{Carlen},~\bfnm{Eric~A.}\binits{E.~A.}},
  \bauthor{\bsnm{Geronimo},~\bfnm{Jeffery~S.}\binits{J.~S.}} \AND
  \bauthor{\bsnm{Loss},~\bfnm{Michael}\binits{M.}}
(\byear{2008}).
\btitle{Determination of the spectral gap in the Kac model for physical
  momentum and energy-conserving collisions}.
\bjournal{SIAM J. Math. Anal.}
\bvolume{40}
\bpages{327--364}.
\bid{mr={2403324}}
\bptok{imsref}%
\end{barticle}
\endbibitem

\bibitem{PDLSC}
\begin{barticle}[auto:STB|2012/03/21|07:41:58]
\bauthor{\bsnm{Diaconis},~\bfnm{Persi}\binits{P.}} \AND
  \bauthor{\bsnm{Saloff-Coste},~\bfnm{Laurent}\binits{L.}}
(\byear{2000}).
\btitle{Bounds for Kac's master equation}.
\bjournal{Comm. Math. Phys.}
\bvolume{209}
\bpages{729--755}.
\bid{mr={1743614}}
\bptok{imsref}%
\end{barticle}
\endbibitem

\bibitem{Janvresse}
\begin{barticle}[auto:STB|2012/03/21|07:41:58]
\bauthor{\bsnm{Janvresse},~\bfnm{E.}\binits{E.}}
(\byear{2001}).
\btitle{Spectral gap for Kac's model of Boltzmann equation}.
\bjournal{Ann. Probab.}
\bvolume{29}
\bpages{288--304}.
\bid{mr={1825150}}
\bptok{imsref}%
\end{barticle}
\endbibitem

\bibitem{LPW}
\begin{bbook}[mr]
\bauthor{\bsnm{Levin},~\bfnm{David~A.}\binits{D.~A.}},
  \bauthor{\bsnm{Peres},~\bfnm{Yuval}\binits{Y.}} \AND
  \bauthor{\bsnm{Wilmer},~\bfnm{Elizabeth~L.}\binits{E.~L.}}
(\byear{2009}).
\btitle{Markov Chains and Mixing Times}.
\bpublisher{Amer. Math. Soc.}, \baddress{Providence, RI}.
\bid{mr={2466937}}
\bptnote{check year}%
\bptok{imsref}%
\end{bbook}
\endbibitem

\bibitem{Maslen}
\begin{barticle}[auto:STB|2012/04/30|08:06:40]
\bauthor{\bsnm{Maslen},~\bfnm{David~K.}\binits{D.~K.}}
(\byear{2003}).
\btitle{The eigenvalues of Kac's master equation}.
\bjournal{Math. Z.}
\bvolume{243}
\bpages{291--331}.
\bid{mr={1961868}}
\bptok{imsref}%
\end{barticle}
\endbibitem

\bibitem{Oliviera}
\begin{barticle}[auto:STB|2012/03/21|07:41:58]
\bauthor{\bsnm{Oliveira},~\bfnm{Roberto~I.}\binits{R.~I.}}
(\byear{2009}).
\btitle{On the convergence to equilibrium of Kac's random walk on matrices}.
\bjournal{Ann. Appl. Probab.}
\bvolume{19}
\bpages{1200--1231}.
\bid{mr={2537204}}
\bptok{imsref}%
\end{barticle}
\endbibitem

\bibitem{PS}
\begin{bmisc}[auto:STB|2012/03/21|07:41:58]
\bauthor{\bsnm{Sidenko},~\bfnm{Sergiy}\binits{S.}}
(\byear{2008}).
\bhowpublished{Kac's random walk and coupon collector's process on poset. Ph.D.
  thesis, Massachusetts Inst. Technology}.
\bid{mr={2717533}}
\bptok{imsref}%
\end{bmisc}
\endbibitem

\bibitem{AS}
\begin{bmisc}[auto:STB|2012/03/21|07:41:58]
\bauthor{\bsnm{Smith},~\bfnm{Aaron}\binits{A.}}
(\byear{2011}).
\bhowpublished{A Gibbs sampler on the $n$-simplex. Available at
  arXiv:\arxivurl{1107.5829}}.
\bptok{imsref}%
\end{bmisc}
\endbibitem

\end{thebibliography}
\end{document}